\documentclass[a4paper, 11pt,reqno]{amsart}

\usepackage{amsfonts}
\usepackage{amssymb}
\usepackage{amsthm}
\usepackage{mathrsfs}
\usepackage{a4wide}
\usepackage{hyperref}
\usepackage{amsmath}
\usepackage{theoremref}
\usepackage[dvipsnames]{xcolor}
\usepackage[utf8]{inputenc}
\usepackage{array}
\usepackage[inline]{enumitem}
\usepackage{comment}

\calclayout

\setlist{topsep=0mm,partopsep=0mm,itemsep=1mm}

\theoremstyle{plain}
\newtheorem{lemma}{Lemma}[section]
\newtheorem{thm}[lemma]{Theorem}
\newtheorem{cor}[lemma]{Corollary}
\newtheorem{prop}[lemma]{Proposition}
\newtheorem*{thm*}{Theorem}

\theoremstyle{definition}
\newtheorem{defn}[lemma]{Definition}
\theoremstyle{remark}
\newtheorem{rem}[lemma]{Remark}

\newcommand{\N}{\mathbb{N}}

\DeclareMathOperator{\supp}{supp}

\newcommand{\qqed}{\hfill\qed}

\newenvironment{thmenumerate}{\begin{enumerate}[label=\textup{(\roman*)},leftmargin=10mm]}{\end{enumerate}}

\makeatletter
\newcommand{\typeitem}[1]{\item[#1]\protected@edef\@currentlabel{#1}}
\makeatother

\makeatletter
\@namedef{subjclassname@2020}{\textup{2020} Mathematics Subject Classification}
\makeatother

\begin{document}

\title[Graph products of residually finite monoids]{Graph products of residually finite monoids are residually finite}

\author{Jung Won Cho}

\address{School of Mathematics and Statistics, University of St Andrews, St Andrews, Fife KY16 9SS, UK.}
\email{jwc21@st-andrews.ac.uk}

\author{Victoria Gould}
\address{Department of Mathematics, University of York, York YO10 5DD, UK}
\email{victoria.gould@york.ac.uk}

\author{Nik Ru\v{s}kuc}

\address{School of Mathematics and Statistics, University of St Andrews, St Andrews, Fife KY16 9SS, UK.}
\email{nik.ruskuc@st-andrews.ac.uk}

\author{Dandan Yang$^1$}

\address{School of Mathematics and Statistics, Xidian University, Xi'an 710071, P. R. China }
\email{ddyang@xidian.edu.cn}

\thanks{The second author is supported by the Engineering and Physical Sciences Research Council Grant EP/V002953/1. The third author is supported by the Engineering and Physical Sciences Research Council Grant EP/V003224/1.
The fourth author is supported by the National Natural Science Foundation of China Grant No. 12171380,  the Natural Science Basic Research Program of Shaanxi Province Grant No. 2023-JC-JQ-04, and  the Shaanxi Fundamental Science Research Project for Mathematics and Physics Grant No. 22JSQ034.}

\thanks{$^1$ The corresponding author.}



\begin{abstract} We show that any  graph product of residually finite monoids is residually finite. As a special case we obtain that any free
product of residually finite monoids is residually finite. The corresponding results for graph products of semigroups follow. \end{abstract}

\maketitle
\section{Introduction}\label{sec:intro}

An
 algebra $A$ is {\em residually finite} if for any $a, b \in A$ with $a \neq b$ there is a finite algebra $B$ and a morphism $\theta \colon A \to B$ such that $a\theta \neq b\theta$.
Residual finiteness is a finitary property for algebras, in that clearly any finite algebra is residually finite. We are concerned here with residual finiteness for monoids, in particular, the preservation of that property under taking free products and, more generally, graph products.

The  free product of a set of disjoint monoids $\mathcal{M}=\{M_\alpha: \alpha\in V\}$ is the monoid freely generated by $\bigcup_{\alpha\in V}M_{\alpha}$.
It contains a copy of each $M_\alpha$, and these copies do not have any non-identity elements in common.
For a simple graph $\Gamma$ with set of vertices $V$ the  graph product of  $\mathcal{M}$ with respect to $\Gamma$ is the monoid freely generated by $\bigcup_{\alpha\in V}M_{\alpha}$, subject to the additional constraints that elements of $M_{\alpha}$ and $M_{\beta}$ commute  whenever there is an edge joining $\alpha$ and $\beta$ in $\Gamma$ \cite{hermiller:1995, dacosta:2001}.
Again,  the graph product contains a copy of each $M_\alpha$, and these copies do not have any non-identity elements in common. We usually refer to the monoids in $\mathcal{M}$ as {vertex monoids}. If the graph has no edges, then the graph product
construction returns  the free product. For finite complete graphs, the graph product simply yields the direct product. We give a formal definition in terms of presentations in Section~\ref{sec:prelim}.

For general reasons, it is true that the direct product $A \times  B$ of two residually finite algebras is residually finite; in fact the residually finite algebras of the same type form a prevariety, the point of view adopted in \cite{sapir:1982}. 
The converse is not true in general, but is for monoids, and indeed whenever $A$ and $B$ are contained in $A\times B$;
the converse is also true for semigroups for somewhat non-obvious reasons \cite{gray:2009}.
The situation for residual finiteness of free products is exactly reverse: if the free product $A\ast B$ is residually finite, then so are $A$ and $B$, because they are contained in $A\ast B$. The converse is not true in general, is true for monoids, but we have not been able to find it in literature. It is a special case of our main theorem in this paper.
Green \cite{Green:1990} showed that graph products of groups are residually finite if and only if each vertex group is residually finite; the proof relies on group-theoretic arguments, not available for arbitrary monoids.

The thrust  of this paper is to prove the following result (see Theorem~\ref{thm:main}).

\begin{thm*}
Any graph product of monoids is residually finite if and only if each vertex monoid is residually finite.\end{thm*}

We present a number of corollaries, including the corresponding result for semigroups.

The structure of this paper is as follows. In Section~\ref{sec:prelim} we set out the background notions needed for this article, including the notion of left Foata normal form for elements of graph products. Section~\ref{sec:blocklength} introduces the notion of block length of words and elements of graph products, and examines the  behaviour of block length under composition.
In Section~\ref{sec:finite} we turn our attention to a special case of our main theorem and show that graph products of finite monoids are residually finite. We then are able to complete our main theorem, and the desired corollaries, in Section~\ref{sec:main}.

\section{Preliminaries}
\label{sec:prelim}

To keep this article as self-contained as possible, here we run through the main concepts required. For more details, we recommend \cite{cliffordpreston:1967,howie:1995}.

\subsection{Monoid presentations}\label{sub:pres}  Let $X$ be a set. The  {\em free monoid} $X^*$ on $X$ consists of all words over $X$ with operation of juxtaposition.
We denote a  non-empty word by $x_1\circ \dots \circ x_n$ where $x_i\in X$ for $1\leq i\leq n$; here $n$ is the {\em length} of the word. We  also
 use $\circ$ for juxtaposition of words. The empty word is denoted by $\epsilon$ and is the identity of $X^*$. Throughout, our convention is that if we say $x_1\circ\dots \circ x_n\in X^*$, then we mean that $x_i\in X$ for all $1\leq i\leq n$, unless we explicitly say otherwise.

A {\em monoid presentation} $\langle X\mid R\rangle$, where $X$ is a set and $R\subseteq  X^*\times X^*$, determines the monoid $X^*/R^{\sharp}$, where $R^{\sharp}$ is the congruence on $X^*$ generated by $R$. In the usual way, we identify $(u,v)\in R$ with the formal equality $u=v$ in a presentation $\langle X\mid R\rangle$. We denote elements of the quotient by $[w]$, where $w\in X^*$.

\subsection{Graph products of monoids}  \label{sub:gp}
In this article, all graphs are simple, that is, they are undirected with no multiple edges or loops.
   Let $\Gamma=(V,E)$ be a  graph;
here $V$ is the  non-empty set of {\em vertices} and  $E$ is  the set of {\em edges} of  $ \Gamma$. We denote an edge joining vertices $\alpha$ and $\beta$ by
$(\alpha,\beta)$ or (since our graph is undirected) by $(\beta,\alpha)$.

Let $\mathcal{M} = \{M_{\alpha} \colon \alpha\in V\}$ be a set of monoids. The {\em direct product} $\prod_{\alpha\in V} M_{\alpha}$ of $\mathcal{M}$ is the monoid where the underlying set consists of all functions $f \colon V \to \bigcup_{\alpha \in V} M_{\alpha}$ such that $\alpha f \in M_{\alpha}$ for all $\alpha \in V$ and the multiplication is defined to be $fg\colon V \to \bigcup_{\alpha\in V} M_{\alpha}$ such that $\alpha (fg) = (\alpha f)(\alpha g)$ for all $\alpha \in V$. The {\em restricted direct product} of $\mathcal{M}$ is the submonoid of $\prod_{\alpha\in V}M_{\alpha}$ consisting of those elements $f \colon V \to \bigcup_{\alpha \in V}M_{\alpha}$ such that $\vert \{\alpha \in V \colon \alpha f \neq 1_{\alpha} \} \vert < \infty$. We note that if $\vert V \vert < \infty$, then the two products coincide.
The definition of the  free product of the monoids in the set $\mathcal{M} = \{M_{\alpha} \colon \alpha\in V\}$, and its universal properties,   may be found in  \cite{cliffordpreston:1967} and \cite{howie:1995}. We do not give details here since free products are  immediately obtained as a special case of the following.

 \begin{defn}[\cite{Green:1990,dacosta:2001}]\label{defn:graphprodmonoids}  Let
$\Gamma=(V,E)$ be  graph and let $\mathcal{M}=\{M_\alpha: \alpha\in V\}$ be a set of mutually disjoint monoids.  For each $\alpha\in V$, we write $1_{\alpha}$ for the identity of $M_{\alpha}$ and put $I=\{ 1_{\alpha}: \alpha\in V\}$.
The {\em graph product} $\mathscr{GP}=\mathscr{GP}(\Gamma,\mathcal{M})$ of
 $\mathcal{M}$   with respect to $\Gamma$ is defined by the monoid presentation
 \[\mathscr{GP}=\langle X\mid R  \rangle\]
 where $X=\bigcup_{\alpha\in V}M_\alpha$ and the relations in $R= R_{\textsf{id}}\cup R_{\textsf{v}}\cup R_{\textsf{e}}$ are given by:
\[\begin{array}{rcl}R_{\textsf{id}}&=&\{   1_\alpha=\epsilon:  \alpha\in V\},\\

R_{\textsf{v}}&=&\{   x\circ y=xy:  \ x,y\in M_\alpha,\alpha\in V\},\\

R_{\textsf{e}}&=& \{ x\circ y=y\circ x: x\in M_\alpha,\, y\in M_\beta, \, (\alpha,\beta)\in E)\}.\end{array}\]
\end{defn}

The monoids $M_{\alpha}$ in Definition~\ref{defn:graphprodmonoids} are  referred to as \textit{vertex monoids}. We may also say for brevity that $\mathscr{GP}(\Gamma,\mathcal{M})$ is the {\em  graph product of the monoids $M_{\alpha},\, \alpha\in V$}. We denote the $R^{\sharp}$-class of $u\in X^*$ by
$[u]$.

\begin{rem}\label{rem:freeprod} Let $\mathcal{M}=\{M_\alpha: \alpha\in V\}$ be a set of mutually disjoint monoids.
\begin{thmenumerate}
\item[$\bullet$]  If $E=\emptyset$, then the  graph product $\mathscr{GP}(\Gamma,\mathcal{M})$ is the  free product of the monoids in $\mathcal{M}$.

\item[$\bullet$] If $E=V\times V$, then the graph product $\mathscr{GP}(\Gamma,\mathcal{M})$  is isomorphic to the  restricted direct product of the monoids in $\mathcal{M}$, and so to the direct product in the case where $|V|<\infty$.

\item[$\bullet$] If the monoids in $\mathcal{M}$ are all free monoids, then $\mathscr{GP}$ is a {\em trace monoid} (see, for example, \cite{diekert:1990}).

\item[$\bullet$] If the monoids in $\mathcal{M}$ are all free groups, then $\mathscr{GP}$ is a {\em free partially commutative group},
or, in the case where $\Gamma$ is finite and each vertex group is finitely generated, {\em a right angled Artin group} (see, for example, \cite{duncan:2020}).
\end{thmenumerate}
\end{rem}

For the remainder of this section  we consider the monoid $\mathscr{GP}=\mathscr{GP}(\Gamma,\mathcal{M})$, and give some of its properties that will be needed in the proof of our main result. We begin with a couple of universal properties.

\begin{lemma}$($\cite[Proposition 1.7]{fountain:2009}$)$ \label{la:morph}
Suppose $\{N_\alpha: \alpha\in V\}$ is another set of mutually disjoint monoids, and  $\phi_\alpha: M_\alpha\rightarrow N_\alpha$, $\alpha\in V$,
are morphisms. Then there is a unique morphism $\phi: \mathscr{GP}(\Gamma, \mathcal{M})\rightarrow \mathscr{GP}(\Gamma, \mathcal{N})$ such that $m_\alpha\phi=m_\alpha\phi_\alpha$ for all $\alpha\in V$ and $m_\alpha\in M_\alpha$. \qqed
\end{lemma}
 
\begin{lemma}$($\cite[Proposition 2.3]{dandan:2023}$)$\label{lem:thedeletiontrick} 
Let $V'\subseteq V$, let $\Gamma'=(V',E')$ be the induced subgraph of  $\Gamma$ on $V'$, and let $\mathcal{M}'=\{ M_\alpha:\alpha\in V'\}$.  Then $\mathscr{GP}(\Gamma',\mathcal{M}')$ is a retract of $\mathscr{GP}(\Gamma,\mathcal{M})$. \qqed
\end{lemma}

\nopagebreak
\subsection{Reduced words}\label{sub:reduced}

\begin{defn}\label{defn:shuffle} We refer to an application of a relation in $R_{\textsf{e}}$ to a word $w\in X^*$ as a {\em shuffle}.
Two words in $X^*$ are {\it shuffle equivalent} if one can be obtained from the other by shuffles.
\end{defn}

\begin{defn}\label{def of reduced} We refer to an application of a relation in $R$ to a word $w\in X^*$ that reduces 
the length of $w$ as a {\em reduction}.
A word $w\in X^*$ is {\em reduced} if it is not possible to apply a reduction step to any word shuffle equivalent to $w$.
\end{defn}

Note that a reduction is either an application of a relation from $R_{\textsf{id}}$ replacing $1_\alpha$ by $\epsilon$, or an application of a relation from $R_\textsf{v}$ replacing $x\circ y$ by $xy$.

The following is key. It appears for groups in \cite{Green:1990}; several subsequent papers have noted the result for monoids e.g. \cite{dacosta:2001,fountain:2009}.
\begin{thm}$($\cite[Theorem 5.5]{dacosta:2001}$)$\label{thm:conf} Every element of $X^*$ is equivalent to a reduced word. If $w,v$ are reduced words, then $w$ is equivalent to $v$ if and only if $w$ and $v$ are shuffle equivalent. \qqed 
\end{thm}

\begin{defn} If $u,w\in X^*$  with $[u]=[w]$ and such that $w$ is reduced, then we say that $w$ is a {\em reduced form} of $u$.\end{defn}

It follows from Theorem~\ref{thm:conf} that all reduced forms of $u\in X^*$ are shuffle equivalent and hence have the same length as words in $X^*$. We now proceed to show that amongst the reduced forms of $u$ we can pick words with useful properties.

\subsection{Foata normal form}

\begin{defn}\label{defn:support}
Let $w=x_1\circ \dots \circ x_n\in X^*$ with $x_i\in M_{\alpha_i}$ for all $1\leq i\leq n$. We call the
set $\{\alpha_1, \dots, \alpha_n\}$ the  {\it support} of $w$ and denote it by $\supp(w)$. 
If $\supp(w)$ is a singleton, i.e. $\supp(w) = \{\alpha\}$, then for the sake of notational convenience we may identify $\supp(w)$ with the vertex $\alpha$.
\end{defn}

\begin{defn}
A reduced word $w\in X^*$ is called a {\it complete block} if the subgraph of $\Gamma$ induced by $\supp(w)$ is complete.
\end{defn}

\begin{rem}
A word $w\in X^*$ is a complete block if and only if $\supp(w)$ is a complete subgraph, no letters in $w$ are from $I$ and for each $\alpha\in \supp(w)$ there exists a unique letter contained in $w$ with support $\alpha$.
\end{rem}

\begin{defn}\label{defn:lcrf} A word  $w\in X^*$ is a  {\it left Foata normal form}
with {\em blocks} $w_i\in X^*$, $1\leq i\leq k$, if the following three conditions are satisfied:
\begin{thmenumerate}
\item\label{it:lcrf1}
$w=w_1\circ \dots \circ w_k$ is a reduced word;
\item\label{it:lcrf2}
$w_i$ is a complete block for all $1\leq i\leq k$;
\item\label{it:lcrf3}
for any $1\leq i<k$ and  $\alpha\in \supp(w_{i+1})$, there is some $\beta\in \supp(w_i)$ such that $(\alpha, \beta)\not\in E$.
\end{thmenumerate}
If $[u]=[w]$ where $w$ is a left  Foata normal form, then we may say
$w$ is a  {\em  left  Foata normal form of $u$}.
Sometimes we will abbreviate the term left  Foata normal form to \emph{LFNF}.
\end{defn}

Note that the empty word $\epsilon$ is a  LFNF by definition.

\begin{thm}[{\cite{dandan:2023}}] \label{thm:uniqueness} Every  element $w$ of $X^*$ has a LFNF  $w_1\circ w_2\circ\dots\circ w_k$,
with blocks $w_i$ for $1\leq i\leq k$.

If  $w_1'\circ w_2'\circ\dots \circ w_h'$ is any LFNF of $w$ with blocks $w_j'$ for $1\leq j\leq h$, then $k=h$ and $[w_i]=[w_i']$
for $1\leq i\leq k$. \qqed
\end{thm}

Theorem~\ref{thm:uniqueness} allows us to define the {\em left block length}  $\ell^l_\Gamma(u)$ of $u\in X^*$ as the number of blocks of any LFNF of $u$; where the graph $\Gamma$ is understood, we write more simply $\ell^l(u)$.
The notion of {\it right Foata normal form} is defined in a dual way and we may define the {\em right block length} $\ell^r_\Gamma(u)$, or $\ell^r(u)$ of $u\in X^*$. In fact, these are equal, and equal to a further useful parameter, as we will show in the next section.

\section{Block length}\label{sec:blocklength}

Throughout this section  we consider a graph product  $\mathscr{GP}=\mathscr{GP}(\Gamma,\mathcal{M})$.

\begin{prop}\label{thm:length_min_no_blcok}  Let $u$ be a reduced word written as a product
$u = u_1 \circ \dots \circ u_n$ such that each $u_i$ is a complete block. Then $\ell^l(u) \leq n$.
\end{prop}

\begin{proof}
We suppose that the assertion is not true, and pick a counter-example $u=u_1\circ\dots\circ u_n$ where $n$ is as small as possible, and, for that $n$, the length of $u_n$ is as small as possible.
Let $p_1\circ\dots \circ p_m$ be a LFNF of $u_1\circ\dots \circ u_{n-1}$ with blocks $p_1,\dots,p_m$. Then $m\leq n-1$.
Consider $p_1\circ\dots \circ p_m\circ u_n$.
If $m<n-1$ we obtain a contradiction by the minimality of $n$.
If $m=n-1$ and there is a letter $x$ in $u_n$ such that $\supp(x)$ is adjacent to all vertices in $\supp(p_m)$, then we have a contradiction to the minimality of $u_n$. 
In the remaining case, $p_1 \circ \dots \circ p_m \circ u_n$ is a LFNF of $u$ and so $\ell^l(u) = n$, which contradicts the choice of counter-example. The result follows.
\end{proof}
The dual to Proposition~\ref{thm:length_min_no_blcok} holds for right block length.

From Proposition~\ref{thm:length_min_no_blcok} and its dual we deduce the main result of this section. 

\begin{thm} \label{thm:blocklength} Let $w\in X^*$ be a reduced word. Then $\ell^l(w)=\ell^r(w)$ and is the least number of complete blocks whose product is equivalent to $w$.
\end{thm}

\begin{defn}\label{defn:block} For any reduced word $w\in X^*$ we write $\ell(w)$ for  $\ell^l(w)=\ell^r(w)$ and refer to  $\ell(w)$ as the {\em block length} of $w$.  More generally, for any  $u\in X^*$ we define  $\ell(u)=\ell(w)$  where $w$ is  any reduced form of $u$, and again  refer to  $\ell(u)$ as the {\em block length} of $u$.
\end{defn}

We note in passing that for the case of graph products of semigroups (which we will consider in Section~\ref{sec:main}) the fact that $\ell^l(w)=\ell^r(w)$  appears in \cite[Lemma 3.8]{yang:2023}.

We now want to examine what happens to the block length of a word when multiplied by a single letter, for which we need a bit of preparation.

\begin{lemma}\label{key3}
Let $x_1\circ \dots \circ x_n\in X^+$ be a reduced word. If there exists $1\leq k\leq n$ such that $(\supp(x_k), \supp(x_j))\in E$ for all $k<j\leq n$, then $x_1\circ \dots \circ x_{k-1}\circ x_{k+1}\circ \dots\circ x_n$ is also reduced.
\end{lemma}

\begin{proof}
Clearly,  $x_1\circ \dots \circ x_n$ is shuffle equivalent to $x_1\circ \dots \circ x_{k-1}\circ x_{k+1}\circ \dots\circ x_n\circ x_k$, implying that $x_1\circ \dots \circ x_{k-1}\circ x_{k+1}\circ \dots\circ x_n\circ x_k$ is reduced. Hence so is $x_1\circ \dots \circ x_{k-1}\circ x_{k+1}\circ \dots\circ x_n$.
\end{proof}

\begin{lemma}\label{key4}
Let $w=w_1\circ \dots \circ w_k$ be a LFNF with blocks $w_i$, $1\leq i\leq k$.
Let $1 \leq j \leq k$ and let $t\in M_{\alpha}$ be a  letter appearing in $w_j$.

\textup{(i)}
Suppose that $(\alpha, \beta)\in E$ for all $\beta\in \supp(w_{j+1}\circ \dots\circ w_k)$, and  let $w_j'$ be obtained from $w_j$ by deleting $t$.
If $w_j'\neq \epsilon$, then  $w_1\circ \dots \circ w_{j-1}\circ w_j'\circ w_{j+1}\circ \dots \circ w_{k}$ is a LFNF with blocks $w_1, \dots,  w_{j-1}, w_j', w_{j+1}, \dots, w_{k}$. 
Further, $w_j'=\epsilon$ can only occur if $j=k$.

\textup{(ii)}
Let $w_j'$ be obtained from $w_j$ by replacing $t$ by any $z\in M_\alpha\setminus\{1_\alpha\}$. Then $w_1\circ \dots \circ w_{j-1}\circ w_j'\circ w_{j+1}\circ \dots \circ w_{k}$ is a LFNF with blocks $w_1, \dots, w_{j-1}, w_j', w_{j+1},\dots, w_{k}$.
\end{lemma}

\begin{proof} We begin by remarking that $t$ is the unique letter of $w_j$ with support $\alpha$.

(i) 
We first verify the second claim. Suppose that $j<k$ and $w_j'=\epsilon$, that is, $w_j=t$. It follows that
\[w=w_1\circ \dots \circ w_{j-1}\circ (t\circ w_{j+1})\circ \dots\circ w_k\]
is a decomposition of $w$ into $k-1$ complete blocks, contradicting Theorem~\ref{thm:blocklength}. Hence, if $w_j'=\epsilon$, then $j=k$.

Now suppose that $w_j'\neq \epsilon$. By Lemma \ref{key3},  deleting the letter $t$, and bearing in mind that $w_j$ is a complete block,  we obtain that  $w_1\circ \dots \circ w_{j-1}\circ w_j'\circ w_{j+1}\circ \dots \circ w_{k}$ is reduced. Since $w$ is a LFNF, for each $\beta\in \supp(w_{j+1})$, there exists $\delta\in \supp(w_j)$ such that $(\beta, \delta)\notin E$. Further, as $(\supp(t), \gamma)\in E$ for all $\gamma \in \supp(w_{j+1}\circ \dots\circ w_k)$, we deduce that $\delta\neq \supp(t)$, so that $\delta\in \supp(w_j')$, implying that $w_1\circ \dots \circ w_{j-1}\circ w_j'\circ w_{j+1}\circ \dots \circ w_{k}$ is a LFNF.

(ii) The proof is immediate.
\end{proof}

We can now state and prove three results which describe the effect on a LFNF and block length of multiplication by a single letter.

\begin{lemma}\label{key2}
Let $w=w_1\circ \dots \circ w_k$ be a LFNF with blocks $w_i$, $1\leq i\leq k$, and let $x\in X\backslash I$.
Then $\ell(w\circ x)=k+1$ if and only if  $\supp(w_k\circ x)$ is an incomplete subgraph, in which case
\[w_1\circ \dots \circ w_k\circ x\]
 is a LFNF for $w\circ x$ with blocks $w_i,\dots,w_k,x$.
\end{lemma}

\begin{proof}
Suppose $\supp(x)=\{\alpha\}$, i.e. $x\in M_\alpha$.
If $\supp(w_k\circ x)$ is incomplete, then $\alpha\notin \supp(w_k)$ and there exists $\beta\in \supp(w_k)$ such that $(\alpha, \beta)\notin E$, implying that $w\circ x$ is reduced. Further,  $w_1\circ \dots \circ w_k\circ x$ is a LFNF  with blocks $w_1,\dots,w_k,x$.
Conversely, if $\supp(w_k\circ x)$ is complete, then from Theorem~\ref{thm:blocklength}, $\ell(w\circ x)\leq k$.
\end{proof}

\begin{lemma}\label{lem:k,k-1}
Let $w=w_1\circ \dots \circ w_k$ be a LFNF with blocks $w_i$, $1\leq i\leq k$, and let $x\in M_\alpha\setminus\{1_\alpha\}$. Suppose that  $\supp(w_k\circ x)$ is a complete subgraph.

\textup{(i)} We have that $\ell(w\circ x)=k-1$ if and only if $\supp(w_k)=\{\alpha\}$ and $w_kx=1_\alpha$ in $M_\alpha$, in which case
$w_1\circ \dots \circ w_{k-1}$
 is a LFNF for $w\circ x$ with blocks $w_i$, $1\leq i\leq k-1$.

\textup{(ii)} Otherwise, $\ell(w\circ x)=k$ and  $w\circ x$ has  a LFNF $w_1\circ \dots \circ w_{i-1}\circ w_i' \circ w_{i+1}\circ \dots \circ w_{k}$
with blocks $w_1,\dots,w_{i-1},w_i',w_{i+1},\dots,w_k$,  where $1\leq i\leq k$ and $w_i'\neq \epsilon$ is a reduced form of $w_i\circ x$.
\end{lemma}

\begin{proof} If $w\circ x$ is not reduced, then there exists $1\leq i\leq k$ and a letter $y$ in $w_i$ such that $y\in M_\alpha$ and $(\alpha, \beta)\in E$ for all $\beta\in \supp(w_{i+1}\circ \dots\circ w_k)$. Let $w_i'$ be a reduced form of $w_i \circ x$. If $yx = 1_{\alpha}$, then $w_i'$ is obtained by deleting $y$ from $w_i$. If $yx\neq 1_{\alpha}$, then $w_i'$ is obtained by replacing $y$ by $yx$ in $w_i$.

By Lemma \ref{key4}, the case $w_i'=\epsilon$ can only occur if $i=k$. In this case,
$w_1\circ \dots \circ w_{k-1}$ is a LFNF for $w\circ x$ with blocks $w_1,\dots,w_{k-1}$. Otherwise,  $w_i'\neq \epsilon$, and from  Lemma \ref{key4}  $w\circ x$ has a LFNF $$w_1\circ \dots \circ w_{i-1}\circ w_i'\circ w_{i+1}\circ \dots \circ w_{k}$$ with blocks $w_1,\dots, w_{i-1},w_i',w_{i+1},\dots, w_k$.

If $w\circ x$ is reduced, then $\alpha\notin \supp(w_k)$. Further, as $\supp(w_k\circ x)$ is complete, we have that $(\alpha, \beta)\in E$ for all $\beta\in \supp(w_k)$. Let $1\leq i\leq k$ be the smallest such that $(\alpha, \beta)\in E$ for all $\beta\in \supp(w_i\circ \dots \circ w_k)$. Then $$w_1\circ \dots \circ w_{i-1}\circ (w_i\circ x) \circ w_{i+1}\circ \dots \circ w_{k}$$ is  a LFNF of $w\circ x$ with blocks 
$w_1, \dots,  w_{i-1}, w_i\circ x, w_{i+1},\dots, w_{k}$.
 \end{proof}

\begin{cor}\label{key6}
Let $w$ be a LFNF of block length $k$ and let $x,y\in M_\alpha\backslash \{1_\alpha\}$. Then $\ell(w\circ x)\leq k$ if and only if $\ell(w\circ y)\leq k$. \qqed
\end{cor}

\section{The finite case} \label{sec:finite}

The main step in the proof of our main theorem is to deal with the special case 
of a graph product $\mathscr{GP}=\mathscr{GP}(\Gamma,\mathcal{M})$ where both the graph $\Gamma$ and all the vertex monoids in $\mathcal{M}$ are finite. This is the subject of this section, and so these assumptions will apply throughout.

In outline,
 we will define actions of the vertex monoids of $\mathscr{GP}$ on a finite set $F_k$. 
 These actions will extend to an action of $\mathscr{GP}$ itself, and hence they
will induce a morphism from $\mathscr{GP}$ to the  finite full transformation monoid on $F_k$. We will then show how to pick a suitable   $k$ so   that the morphism separates pairs of distinct elements of $\mathscr{GP}$ having block length no greater than $k$.
 
Before we embark on this, we briefly review the basic definitions and notation for actions.

\begin{defn}\label{defn:monoidactions}
Let $M$ be a monoid and let $A$ be a non-empty set. An {\em action of } $M$  on $A$ is a map $*:A\times M\rightarrow A$, where $(a,m)\mapsto a*m$, such that for all $a\in A$ and $s,t\in M$ we have $a*1=a$ and $(a*s)*t=a*st$.
\end{defn}

An action of $M$ on $A$ corresponds precisely  to a monoid morphism $M\rightarrow \mathcal{T}_A$, where $\mathcal{T}_A$ is the full transformation monoid of all maps from $A$ to $A$. Further details may be found in the standard text \cite{kkm:2000}.

\begin{lemma}\label{lem:action}Let $k\in \mathbb{N}$ and let
\begin{align*}
    F_{k} = \{[w]\in \mathscr{GP}\colon \ell(w) \leq k\}.
\end{align*}
Each vertex monoid $M_\alpha$ acts on $F_k$ via:
\[
    [w]*x =\begin{cases}[w\circ x]&\textup{if }\ell(w\circ x)\leq k,\\
{}[w]&\textup{if } \ell(w\circ x)=k+1.\end{cases}
\]
\end{lemma}

\begin{proof}
Let $m\leq k$ and let $w=w_1\circ \dots \circ w_m$ be a LFNF with blocks $w_i$, $1\leq i\leq m$.

Clearly,  $[w\circ 1_\alpha]=[w]$,  so that $\ell(w\circ x)=\ell(w)\leq k$, and hence
$$[w]*1_\alpha=[w\circ 1_\alpha]=[w].$$

Let $x, y\in M_\alpha$. We must show that $([w]*x)*y=[w]*xy$.

If $x=1_{\alpha}$ or $y = 1_{\alpha}$, then the claim is clearly true.
 Assume therefore that $x, y\neq 1_{\alpha} $.

By Lemmas~\ref{key2} and ~\ref{lem:k,k-1} we have that $\ell(w\circ x)\leq m+1$. Hence, if
$m\leq k-1$, then $\ell(w\circ x), \ell(w\circ y)$ and $\ell(w\circ x\circ y)=\ell(w\circ xy)$ are all no greater that $m+1\leq k$, so that in this case,
$$([w]\ast x)*y=[w\circ x]*y=[w\circ x\circ y]=[w\circ (xy)]=[w]*(xy).$$

Suppose that $m=k$. If $\ell(w\circ x)=k+1$, then by Corollary~\ref{key6} we  have that $\ell(w\circ y)=k+1$, and if $xy\neq 1_{\alpha}$, then  also $\ell(w\circ xy)=k+1$.
Hence, whether or not we have $xy\neq 1_{\alpha}$, we deduce
\[([w]*x)*y=[w]*y=[w]=[w]*(xy).\qedhere\]
\end{proof}

The above action of each $M_{\alpha}$ induces a morphism from $M_{\alpha}$ to $\mathcal{T}_{F_k}$. We now aim to show that this induces  a morphism from the entire
$\mathscr{GP}$ to  $\mathcal{T}_{F_k}$. 
The key step is the following lemma, which shows that the defining relations $R_{\textsf{e}}$ are respected by the actions of the~$M_\alpha$.

With some abuse of notation we use $*$ for the action of any $M_{\alpha}$ on $F_k$.
It will always be clear from the context to which specific action $*$ refers.

\begin{lemma}\label{lem:thecrux} If $x\in M_{\alpha}$ and $y\in M_{\beta}$ and $(\alpha,\beta)\in E$, then
\[([w]*x)*y=([w]*y)*x\] for all $[w]\in F_k.$
\end{lemma}

\begin{proof}
Without loss of generality we may assume that $w=w_1\circ \dots \circ w_m$ is a LFNF with blocks $w_i$, $1\leq i\leq m\leq k$.  Clearly, the result is true if either $x=1_\alpha$ or $y=1_\beta$.

Suppose that $x, y\notin I$.
We split our considerations into cases depending on $\ell(w\circ x)$ and $\ell(w\circ y)$.

{\em Case 1: $\ell(w\circ x)=k+1$ and $\ell(w\circ y)=k+1$.} We have $[w]*x=[w]$ and $[w]*y=[w]$, and so $$([w]*x)*y=[w]*y=[w]=[w]*x=([w]*y)*x.$$

{\em Case 2: $\ell(w\circ x)\leq k$ and $\ell(w\circ y)=k+1$.} This time $[w]*x=[w\circ x]$ and $[w]*y=[w]$.
 It follows from  Lemmas~\ref{key2} and ~\ref{lem:k,k-1}
that $m=\ell(w)=k$ and that
$\supp(w_k\circ y)$ is incomplete.
If $\ell(w\circ x)=k-1$, then, by Lemma ~\ref{lem:k,k-1}, we have $\supp(w_k)=\{\alpha\}$ and $w_kx=1_\alpha$. 
On the other hand, as $\ell(w\circ y)=k+1$, we have that $\supp(w_k\circ y)$ is incomplete, again, by Lemma~\ref{lem:k,k-1}. But this contradicts the assumption that $(\alpha,\beta)\in E$. Therefore, $\ell(w\circ x)=k$.
It follows from Lemma~\ref{lem:k,k-1}  that $w\circ x$ has a LFNF $$w_1\circ \dots \circ w_{i-1}\circ (w_i\circ x)' \circ w_{i+1}\circ \dots \circ w_{k}$$
 $(w_i\circ x)'\neq \epsilon$ is a reduced form of $w_i\circ x$ for some $1\leq i\leq k$. If $i<k$, then we know
 $\supp(w_k\circ y)$ is incomplete. 
 If $i=k$, then $\supp(w_k')$ is equal to one of $\supp(w_k)$, $\supp(w_k)\cup \{\alpha\}$ or $\supp(w_k)\setminus \{\alpha\}$; but as
 $(\alpha,\beta)\in E$ we have in any of these cases that $\supp(w_k'\circ y)$ is incomplete. It follows that  $\ell(w\circ x\circ y)=k+1$. Then
\[([w]*x)*y=[w\circ x]*y=[w\circ x]=[w]*x=([w]*y)*x.\]

{\em Case 3: $\ell(w\circ x)= k+1$ and $\ell(w\circ y)\leq k$.} This is dual to Case 2.

{\em Case 4: $\ell(w\circ x)\leq k$ and $\ell(w\circ y)\leq k$.}  We have $[w]*x = [w\circ x]$ and $[w]*y = [w\circ y]$.
We claim that
\begin{equation}
\label{eq:lwyx}
\ell(w\circ x \circ y) = \ell (w\circ y \circ x) \leq m+1 \leq k.
\end{equation}

To see this, first suppose
 that $m < k$. Then  $[w\circ x \circ y] = 
[w_1 \circ \dots \circ w_m \circ (x\circ y)]=[ w_1\circ \dots w_m \circ (y\circ x)]= [w \circ y \circ x]$.  Certainly $\supp(x\circ y)$ is a complete block since   
$(\alpha,\beta) \in E$,
 and \eqref{eq:lwyx} follows by 
Theorem~\ref{thm:blocklength}.
Next, suppose that $m = k$. By Lemma~\ref{lem:k,k-1} we know that  $\supp(w_k\circ x)$ and $\supp(w_k\circ y)$ are complete graphs. Then, again since $(\alpha,\beta)\in E$, we have that  $\supp(w_k\circ x\circ y)$ and $\supp(w_k\circ y \circ x)$ are complete graphs, so their (common) reduced form is a complete block. 
Hence \eqref{eq:lwyx} follows using  Theorem~\ref{thm:length_min_no_blcok}.

Having established \eqref{eq:lwyx}, we have
\[([w]*x)*y = [w\circ x]*y = [w \circ x \circ y] = [w \circ y \circ x] = [w\circ y] * x = ([w]*y)*x,\]
completing the proof of this case and the lemma.
\end{proof}

Let us denote by $\theta_\alpha$ the morphism $M_{\alpha}\rightarrow \mathcal{T}$ induced by the action of $M_\alpha$ on $F_k$;
thus 
\[
[w](x\theta_{\alpha})=[w]*x\quad ([w]\in F_k,\  x\in M_{\alpha}).
\]

We let $\theta:X\rightarrow \mathcal{T}_{F_k}$ be $\bigcup_{\alpha\in V}\theta_\alpha$.

\begin{prop}\label{prop:action} Let $k\in\N$. Then there is a morphism $\Theta:\mathscr{GP}\rightarrow \mathcal{T}_{F_k}$ given by
\[[w]([u]\Theta)=[w](x_1\theta)(x_2\theta)\dots (x_m\theta)\,\,\mbox{ where }u=x_1\circ x_2\circ \dots\circ x_m\in X^*.\]
 \end{prop}
 \begin{proof} Since $X^*$ is free on $X$, the map $\theta$ lifts to a morphism $\theta':X^*\rightarrow F_k$, where
 $(x_1\circ x_2\circ \dots\circ x_m)\theta'=(x_1\theta)(x_2\theta)\dots (x_m\theta)$. To show that $\theta'$ induces $\Theta$ as given, we need only show that
 the relations $R=R_{\textsf{id}}\cup R_{\textsf{v}}\cup R_{\textsf{e}}\subseteq \ker \theta'$. The fact that $R_{\textsf{id}}\cup R_{\textsf{v}}\subseteq \ker \theta'$ follows from Lemma~\ref{lem:action}
 and the fact that $R_{\textsf{e}}\subseteq \ker \theta'$ from Lemma~\ref{lem:thecrux}.
\end{proof}

Note that in the language of actions, Proposition~\ref{prop:action} says that
$\mathscr{GP}$ acts on $F_k$ by
\[
[w]*[u]=[w]*x_1*x_2*\dots  *x_n,\,\,\mbox{ where }u=x_1\circ x_2\circ \dots\circ x_m\in X^*.
\]

\begin{prop}\label{main_prop_graph_products}
Any graph product of finite monoids with respect to a finite graph is residually finite.
\end{prop}

\begin{proof}  Let $\mathscr{GP}$  be a finite graph product of finite monoids. Let $[u]$ and $[v]$ be distinct elements of $\mathscr{GP}$; we may assume that $u$ and $v$ are reduced. Choose $k$ to be the greater of $\ell(u)$ and $\ell(v)$ and consider the morphism $\Theta:\mathscr{GP}\rightarrow\mathcal{T}_{F_k}$ from
Proposition~\ref{prop:action}. It follows from Theorem~\ref{thm:blocklength} that
for any $w\in X^*$ where $u=w\circ t$ or $v=w\circ t$, we have that $\ell(w)\leq k$. Consequently, writing $u=u_1\circ\dots \circ u_h\in X^*$, we see that
\begin{align*}
[\epsilon] ([u]\Theta)&= [\epsilon] (u_1\theta\circ \dots \circ u_h\theta)=[\epsilon]*u_1*u_2*\dots  *u_h\\
&=[u_1]*u_2*\dots  *u_h
=[u_1\circ u_2]*\dots  *u_h
=\dots
=[u_1\circ\dots \circ u_h]= [u].
\end{align*}

Similarly, $[\epsilon] ([v]\Theta)=[v]$ and since $[u]\neq [v]$ we obtain that $[u]\Theta\neq [v]\Theta$, as required.
\end{proof}

\section{The main results}\label{sec:main}

We can now prove our main result.

\begin{thm}\label{thm:main}
Any graph product of monoids is residually finite if and only if each vertex monoid is residually finite. 
\end{thm}

\begin{proof} 
Let $\mathscr{GP} = \mathscr{GP}(\Gamma, \mathcal{M})$ be a graph product of monoids, with $\Gamma=(V,E)$ and $\mathcal{M}=\{ M_\alpha\colon \alpha\in V\}$.
As already mentioned in the Introduction, the `only if' part follows from the fact that all $M_\alpha$ embed into $\mathscr{GP}$.
For the `if' part, suppose that all $M_\alpha$ are residually finite.
Let
$[v]$ and $[w]$ be distinct elements of $\mathscr{GP}$.
We may assume that $v=x_1 \circ \dots \circ x_n, w=y_1 \circ \dots \circ y_m\in X^*$ are reduced. 
Let $V'=\supp(v)\cup \supp(w)$ and let $Y=\{ x_1,\dots, x_n,y_1,\dots, y_m\}$. For each $\alpha\in V'$ the residual finiteness of $M_{\alpha}$ allows us choose a finite monoid $M_{\alpha}'$ and a morphism $\theta_{\alpha}:M_{\alpha}\rightarrow M'_{\alpha}$ which separates the distinct elements of $(M_{\alpha}\cap Y)\cup \{ 1_{\alpha}\}$. For $\alpha\notin V'$ we let
$M_{\alpha}'=M_{\alpha}$ and let $\theta_{\alpha}:M_{\alpha}\rightarrow M_{\alpha}$ be the identity morphism.
Let $\mathscr{GP'} = \mathscr{GP}(\Gamma, \mathcal{M'})$ where $\mathcal{M'} = \{M'_{\alpha} \colon \alpha \in V\}$. 
By Lemma~\ref{la:morph},  $\mathscr{GP'}$ is a  morphic image of $\mathscr{GP}$ via the morphism $\theta$ which extends each $\theta_{\alpha}$.

We claim that $[v]\theta\neq [w]\theta$.
Clearly, $[v]\theta=[x_1\theta\circ\dots \circ x_n\theta]$, $[w]\theta=[y_1\theta\circ\dots \circ y_m\theta]$ 
and 
$x_1\theta\circ\dots \circ x_n\theta$, $y_1\theta\circ\dots \circ y_m\theta$ are reduced words.
If $ [v]\theta=[w]\theta$, then we would be able to shuffle $x_1\theta\circ\dots \circ x_n\theta$ to $y_1\theta\circ\dots \circ y_m\theta$ by Theorem~\ref{thm:conf}. Hence, by a corresponding sequence of moves, we could shuffle $x_1 \circ \dots \circ x_n$ to $y_1 \circ \dots \circ y_m$, a contradiction. Therefore, $[v]\theta\neq [w]\theta$, as claimed.

Let  $\Gamma'=(V',E')$ be  the induced subgraph on $V'$. By Lemma \ref{lem:thedeletiontrick} there exists a retraction $\phi:\mathscr{GP}'\rightarrow  \mathscr{GP''}$ , where $\mathcal{M''}=\{ M_\alpha':\alpha\in V'\}$ and $\mathscr{GP''}=\mathscr{GP}(\Gamma', \mathcal{M''})$. By construction, we have $[v]\theta\phi\neq [w]\theta\phi$.

Note that $\mathscr{GP''}$ is a graph product of finite monoids with respect to a finite graph.
Hence it is residually finite by Proposition \ref{main_prop_graph_products}.
Thus there exist a finite monoid $M$ and a morphism $\psi:\mathscr{GP''}\rightarrow M$ such that $[v]\theta\phi\psi\neq [w]\theta\phi\psi$.
But then the composition $\theta\phi\psi$ is a morphism from our original graph product $\mathscr{GP}$ into the finite monoid $M$ separating $[v]$ and $[w]$, and the theorem is proved.
\end{proof}

We record the following immediate consequence:

\begin{cor}
A free product or a restricted direct product of  monoids is
residually finite if and only if each vertex monoid is residually finite. \qqed
\end{cor}

As  commented in the Introduction, the result for finite direct products follows from generic considerations.
The result for free products does not seem to have appeared in literature.
The result for residual finiteness of graph products of groups appears in \cite{Green:1990}; in contrast to the argument there, our proof for monoids does not
require any tools specific for groups.

We now briefly consider the  notion of a graph product of semigroups (see \cite{dandan:2023}). Here we start with a {\em semigroup presentation}, which is a quotient of the free {\em semigroup} $X^+=X^*\setminus\{ \epsilon\}$ on a set $X$.

 \begin{defn}\label{defn:graphprodsemigroups}  Let
$\Gamma=(V,E)$ be  graph and let $\mathcal{S}=\{S_\alpha: \alpha\in V\}$ be a set of mutually disjoint semigroups.
The {\em graph product} $\mathscr{GP}=\mathscr{GP}(\Gamma,\mathcal{S})$ of
 $\mathcal{S}$   with respect to $\Gamma$ is defined by the semigroup presentation
 \[\mathscr{GP}=\langle X\mid R  \rangle\]
 where $X=\bigcup_{\alpha\in V}S_\alpha$ and the relations in $R=  R_{\textsf{v}}\cup R_{\textsf{e}}$ are given as in Definition \ref{defn:graphprodmonoids}.
\end{defn}

\begin{cor}\label{cor:semigroups} Any graph product of  semigroups is residually finite if and only if each vertex semigroup is residually finite. 
\end{cor}
\begin{proof}  Consider a graph product $\mathscr{GP}=\mathscr{GP}(\Gamma,\mathcal{S})$ of semigroups
$\mathcal{S}=\{S_\alpha: \alpha\in V\}$. For each $\alpha\in V$ we may embed $S_{\alpha}$ into a monoid $S_{\alpha}^{\underline{1}_{\alpha}}$ where $\underline{1}_{\alpha}$ is an adjoined identity. From \cite[Proposition 7.3]{dandan:2023}, there is a semigroup embedding of
$\mathcal{GP}$ into $\mathcal{GP}'=\mathscr{GP}(\Gamma,\mathcal{M})$ where $\mathcal{M}=\{ S_{\alpha}^{\underline{1}_{\alpha}}:\alpha\in V\}$.

If each semigroup $S_{\alpha}$ is residually finite, then so is each monoid $S_{\alpha}^{\underline{1}_{\alpha}}$ and hence from Theorem~\ref{thm:main}, so is
$\mathcal{GP}'$. But, $\mathcal{GP}$ embeds into $\mathcal{GP}'$, so that $\mathcal{GP}$ is residually finite.

Conversely, if $\mathcal{GP}$ is residually finite, then as each $S_{\alpha}$ embeds into $\mathcal{GP}$, so is each $S_{\alpha}$.
\end{proof}

Again, we obtain the following immediate consequence:

\begin{cor}[Golubov \cite{golubov:1971}]
A free product of semigroups is residually finite if and only if all the vertex semigroups are residually finite.\qqed
\end{cor}

We also remark that a finite direct product of semigroups is not realised by a graph product, and the notion of restricted direct product cannot be formulated, since there are no given identities. We observe, however, that it is shown in \cite{mayr:2018} that a finite direct product of semigroups is residually finite if and only if each constituent semigroup is residually finite.

\section*{Acknowledgement} 

The authors are grateful to Prof. Mikhail Volkov for alerting them to the existence of \cite{sapir:1982} and for providing them with references 
\cite{golubov:1971} and {\cite{sapir:1982}.

\end{document}